\pdfoutput=1
\RequirePackage{ifpdf}
\ifpdf 
\documentclass[pdftex]{sigma}
\else
\documentclass{sigma}
\fi

\numberwithin{equation}{section}

\usepackage{csquotes} 
\usepackage{mathtools} 
\usepackage{stmaryrd} 
\usepackage{physics} 
\usepackage{tensor} 
\usepackage{slashed} 
\usepackage{faktor} 
\usepackage{mathrsfs} 
\usepackage{bbm} 
\usepackage{tikz} 
\usetikzlibrary{cd}
\usepackage{booktabs} 
\usepackage{colortbl} 
\usepackage{array} 
\usepackage{makecell} 
\usepackage{multirow} 
\usepackage{adjustbox} 
\usepackage{enumitem} 
\setenumerate[1]{label={(\arabic*)}}
\newtheorem{Theorem}{Theorem}[section]
\newtheorem*{Theorem*}{Theorem}
\newtheorem{Corollary}[Theorem]{Corollary}
\newtheorem{Lemma}[Theorem]{Lemma}
\newtheorem{Proposition}[Theorem]{Proposition}
 { \theoremstyle{definition}
\newtheorem{Definition}[Theorem]{Definition}

 }
\newcommand{\fa}{\mathfrak{a}}

\newcommand{\fh}{\mathfrak{h}}

\newcommand{\fs}{\mathfrak{s}}

\newcommand{\fso}{\mathfrak{so}}

\newcommand{\fX}{\mathfrak{X}}
\newcommand{\fK}{\mathfrak{K}}

\newcommand{\fS}{\mathfrak{S}}
\newcommand{\fV}{\mathfrak{V}}
\newcommand{\fD}{\mathfrak{D}}

\newcommand{\RR}{\mathbb{R}}
\newcommand{\CC}{\mathbb{C}}
\newcommand{\HH}{\mathbb{H}}
\newcommand{\ZZ}{\mathbb{Z}}

\newcommand{\bbS}{\mathbb{S}}

\newcommand{\eL}{\mathscr{L}}

\newcommand{\cH}{\mathcal{H}}

\newcommand{\ssB}{\mathsf{B}}
\newcommand{\ssC}{\mathsf{C}}
\newcommand{\ssH}{\mathsf{H}}
\newcommand{\ssZ}{\mathsf{Z}}

\newcommand{\ssS}{\mathsf{S}}
\newcommand{\Sbundle}{\underline{S}}

\newcommand{\fstilde}{\widetilde{\fs}}

\newcommand{\Wedge}{\mathchoice{{\textstyle\bigwedge}}{\bigwedge}{\bigwedge}{\bigwedge}}	
\newcommand{\Odot}{\mathchoice{{\textstyle\bigodot}}{\bigodot}{\bigodot}{\bigodot}}
\newcommand{\Otimes}{\mathchoice{{\textstyle\bigotimes}}{\bigotimes}{\bigotimes}{\bigotimes}}

\newcommand{\be}{\boldsymbol{e}}

\newcommand{\1}{\mathbbm{1}}

\newcommand{\pair}[2]{\left\langle #1,#2\right\rangle} 
\DeclareMathOperator{\Id}{Id}

\DeclareMathOperator{\End}{End}
\DeclareMathOperator{\Hom}{Hom}

\DeclareMathOperator{\Spin}{Spin}

\DeclareMathOperator{\Cl}{Cl}

\DeclareMathOperator{\vol}{vol}

\begin{document}
\allowdisplaybreaks

\newcommand{\arXivNumber}{2410.01765}

\renewcommand{\PaperNumber}{082}

\FirstPageHeading

\ShortArticleName{Killing Superalgebras in 2 Dimensions}

\ArticleName{Killing Superalgebras in 2 Dimensions}

\Author{Andrew D.K. BECKETT}

\AuthorNameForHeading{A.D.K.~Beckett}

\Address{University of Edinburgh, UK}
\Email{\href{mailto:adkbeckett@proton.me}{adkbeckett@proton.me}}
\URLaddress{\url{https://adkbeckett.github.io/}}

\ArticleDates{Received November 28, 2024, in final form September 15, 2025; Published online September 30, 2025}

\Abstract{We provide some examples of Killing superalgebras on 2-dimensional pseudo-Riemannian manifolds within the theoretical framework established in [\textit{SIGMA} \textbf{21} (2025), 081, 61~pages]. We compute the Spencer cohomology group $\mathsf{H}^{2,2}(\mathfrak{s}_-;\mathfrak{s})$ and filtered deformations of the non-chiral flat model (Euclidean and Poincar\'e) superalgebra $\mathfrak{s}$ for various Dirac currents and show these arise as Killing superalgebras for (imaginary) geometric Killing and skew-Killing spinors in both Riemannian and Lorentzian signature.}


\Classification{17B66; 17B56}


\section{Introduction}

In recent work \cite{Beckett2024_2}, we discussed a generalisation of the concept of Killing superalgebras, which arise as supersymmetry algebras of supersymmetric solutions in supergravity theories \cite{Figueroa-OFarrill2007_1}, to a~more general context. To this end, we introduced the concept of an \emph{admissible connection} $D$ on a bundle of spinors $\Sbundle$ over a pseudo-Riemannian manifold $M$ of arbitrary signature equipped with a Dirac current $\kappa\colon \Odot^2 \Sbundle\to TM$ or $\kappa\colon \Wedge^2 \Sbundle\to TM$. The Killing (super)algebra associated to this data was then defined as the vector space $\fK_D=\fV_D\oplus\fS_D$, where $\fS_D$ is the space of $D$-parallel spinor fields (known here as \emph{Killing spinors}) and $\fV_D$ is the space of Killing vectors preserving the connection equipped with a bracket defined using the Dirac current $\kappa$ and the (vectorial and spinorial) Lie derivative \cite[Definition~3.6]{Beckett2024_2}.

We discussed the algebraic structure of these Killing superalgebras, showing \cite[Theorem~3.10]{Beckett2024_2} that they were filtered subdeformations of \emph{flat model} (super)algebras which are generalisations of the ($N$-extended) Poincar\'e superalgebra first systematically discussed in \cite{Strathdee1986}. The classification of these flat models in general signature is given in \cite{Alekseevsky1997}. We also discussed how the Spencer cohomology of the flat models can be used to study their filtered subdeformations with particular focus on the case of Lorentzian signature and symmetric, causal Dirac current, generalising many results from the previously-studied 11-dimensional supergravity case \cite{Figueroa-OFarrill2017_1} to general dimension and $N$-extension.

While providing a number of theoretical results, our previous work \cite{Beckett2024_2} did not contain any concrete examples. Numerous examples are provided by existing work in Lorentzian signature, primarily in 11 dimensions \cite{Figueroa-OFarrill2016,Figueroa-OFarrill2017_1,Figueroa-OFarrill2017,Santi2022} but also in 4, 5 and 6 dimensions \cite{Beckett2019,deMedeiros2016,deMedeiros2018}, while work in Riemannian signature \cite{Figueroa-OFarrill2008_except} has shown that geometric Killing spinors on higher-dimensional spheres provide geometric realisations of the exceptional algebras $\mathfrak{f}_4$ and $\mathfrak{e}_8$ as well as $\fso(8)$ triality as Killing algebras. However, the literature is lacking in some simpler examples which demonstrate the framework as well as the effect of different choices -- for example, signature and Dirac current -- on the definition of admissible connections and on the associated Killing (super)algebras. In this work, we provide these simple examples by considering admissible connections and Killing (super)algebras over 2-dimensional manifolds, and we also study the Spencer cohomology group $\cH^{2,2}$ and filtered deformations of the relevant flat models.

This work is intended as a companion to \cite{Beckett2024_2} which illustrates its main theoretical ideas and is best read in parallel with that work; we use the same terminology and conventions in both works. We will consider only signatures $(0,2)$ and $(1,1)$ since the Clifford algebra is isomorphic to the real matrix algebra $\RR(2)$ in both cases, hence the situation is much simpler than the $(2,0)$ case where the Clifford algebra is $\HH$. Indeed, the former cases are so similar that we will be able to perform many of the calculations in a signature-agnostic way; we will however choose a sign convention (see equation \eqref{eq:clifford-rel}) which will allow us to work in practice with a positive-definite inner product in signature $(0,2)$. The quaternionic $(2,0)$ case, as well as other generalisations, will be treated in future work.

\section{Constructions, conventions and formulae for spinor modules}

Throughout, we let $V=\RR^{1,1}$ or $\RR^{0,2}$ and note as mentioned in the introduction that in either case $\Cl(V)\cong\RR(2)$. We will define explicit matrix representations of the Clifford algebras in either case below; for now we note that these will allow us to identify the real irreducible pinor module $\bbS$ as $\RR^2$ under left multiplication by matrices.

\subsection{Signature-agnostic Clifford algebra conventions}
\label{sec:gamma-matrix}

As discussed in \cite[Appendix~A.1.2]{Beckett2024_2}, there is a choice of sign convention in the Clifford algebra relation
\begin{equation}\label{eq:clifford-rel}
	v\cdot v = \pm\eta(v,v)\1,
\end{equation}
where $v\in V$; we take the mathematical convention with $-$ in defining $\Cl(p,q)=\Cl(\RR^{p,q})$ but the more common convention in physics with $+$ for explicit calculations, following some previous work on Killing superalgebras. We will speak of signature $(0,2)$ to indicate that we are working with the Clifford algebra $\Cl(0,2)$, thus with the $+$ convention above the inner product $\eta$ must in fact be \emph{positive-definite}. In choosing (ordered) orthonormal bases $\{e_\mu\}$ for $V$, we take the ``\emph{mostly-positive}'' convention $\eta_{00}=-\eta_{11}=-1$ in signature $(1,1)$. We define the sign
$
	\varsigma=\det[\eta]
$
which is $+1$ in signature $(0,2)$ and $-1$ in signature $(1,1)$.

In an orthonormal basis $\{e_\mu\}$, we denote the matrix representing $e_\mu$ under any specified isomorphism $\Cl(V)\to\RR(2)$ by $\Gamma_\mu$ and also define $\Gamma_{\mu\nu}:=\Gamma_{[\mu}\Gamma_{\nu]}$. The matrix representing the canonical volume element $\vol\in\Wedge^2 V$ is $\Gamma_*=\Gamma_{12}$ or $\Gamma_{01}$ depending on signature. Respectively to signature, $\{\1,\Gamma_1,\Gamma_2,\Gamma_*\}$ or $\{\1,\Gamma_0,\Gamma_1,\Gamma_*\}$ is a basis for $\RR(2)$ and we will refer to these as \emph{$\Gamma$-matrices}. The algebraic relations among these matrices are as follows:
\[
	\Gamma_{\mu}\Gamma_{\nu} = \Gamma_{\mu\nu} + \eta_{\mu\nu}\1 = \varepsilon_{\mu\nu}\Gamma_* + \eta_{\mu\nu}\1,
	\qquad	\Gamma_\mu \Gamma_* = -\Gamma_* \Gamma_\mu = \varsigma \varepsilon_{\mu\nu}\Gamma^\nu,
	\qquad (\Gamma_*)^2 = -\varsigma\1,
\]
where $\varepsilon_{\mu\nu}$ is the Levi-Civita symbol defined so that $\varepsilon_{12}=-\varepsilon_{21}=+1$ in definite signature and $\varepsilon_{01}=-\varepsilon_{10}=+1$ in Lorentzian signature. Using the metric to raise and lower indices, we must define the symbols with raised indices so that $\varepsilon^{12}=-\varepsilon^{21}=+1$ and $\varepsilon^{01}=-\varepsilon^{10}=-1$ respectively. We have the following combinatorial identities for $\varepsilon_{\mu\nu}$:
\[
	\varepsilon_{\mu\nu}\varepsilon_{\rho\sigma}
		= \varsigma\qty(\eta_{\mu\rho}\eta_{\nu\sigma} - \eta_{\mu\sigma}\eta_{\nu\rho}),
	\varepsilon_{\mu\nu}\varepsilon\indices{_\rho^\nu}
		= \varsigma\eta_{\mu\rho},\qquad
	\varepsilon_{\mu\nu}\varepsilon^{\mu\nu} = 2\varsigma.
\]
For the avoidance of confusion in calculations, it is useful to denote the top-rank gamma matrix with indices raised as $\Gamma^*=\Gamma^{12}$ or $\Gamma^{01}$ and note the following
$\Gamma^* = \varsigma\Gamma_*$,
$\Gamma^{\mu\nu} = \varepsilon^{\mu\nu}\Gamma_* = \varsigma\varepsilon^{\mu\nu}\Gamma^*$.
We also have the following traces of products
$\Gamma^\mu\Gamma_\mu = 2\1$, $
	\Gamma^\mu\Gamma_\nu\Gamma_\mu = 0$, $
	\Gamma^\mu\Gamma_*\Gamma_\mu = -2\Gamma_*$.
Finally, we note that the even subalgebra $\Cl_{\overline 0}(V)$ corresponds to the span $\RR\{\1,\Gamma_*\}$ and the Lie algebra~${\fso(V)\cong\RR}$ to the span $\RR\Gamma_*$.

\subsection{Pinor and spinor modules, admissible bilinears and Dirac currents}
\label{sec:2d-dirac-currents}

We now describe the (s)pinor modules as well as their admissible bilinears and Dirac current explicitly. We will make extensive use of the Pauli matrices
\[
	\sigma_1 = \begin{pmatrix} 0 & 1 \\ 1 & 0 \end{pmatrix},
	\qquad\sigma_2 = \begin{pmatrix} 0 & -{\rm i} \\ {\rm i} & 0 \end{pmatrix},
	\qquad \sigma_3 = \begin{pmatrix} 1 & 0 \\ 0 & -1 \end{pmatrix}.
\]
Note that $\sigma_1$, $\sigma_3$ are real symmetric matrices and $\Omega = {\rm i}\sigma_2$ is the (real and skew-symmetric) standard symplectic matrix. The sets $\{\1,\sigma_1,\sigma_2,\sigma_3\}$ and $\{\1,\sigma_1,\sigma_3,\Omega={\rm i}\sigma_2\}$ are both $\CC$-bases for~$\CC(2)$, while the latter is also an $\RR$-basis for $\RR(2)$. The Pauli matrices satisfy the algebraic identities
\smash{$\sigma_i^\dagger = \sigma_i$},
$ \sigma_i \sigma_j = \delta_{ij}\1 + {\rm i}\varepsilon_{ijk}\sigma_k$,
where \smash{${}^\dagger$} denotes Hermitian adjoint (conjugate-transpose), $\varepsilon_{ijk}$ is the Levi-Civita symbol with normalisation $\varepsilon_{123}=+1$, and we use the Einstein summation convention on repeated indices. The second identity is also useful to express in the form~${\comm{\sigma_i}{\sigma_j} = 2{\rm i}\varepsilon_{ijk}\sigma_k}$,
$\acomm{\sigma_i}{\sigma_j} = 2\delta_{ij}\1$,
where $\comm{-}{-}$ is the matrix commutator, $\acomm{-}{-}$ is the anti-commutator.

We recall the following from \cite{Alekseevsky1997} where the flat model (super)algebras are classified by classifying the possible Dirac currents, that is $\fso(V)$-equivariant maps $\kappa\colon\Odot^2S\to V$ where $S$ is a~(possibly $N$-extended) spinor module of $\fso(V)$. This classification essentially reduces to classifying such maps on irreducible pinor modules $\ssS$, and the space of such maps is spanned by those Dirac currents $\kappa$ obtained from admissible bilinears $B$ (defined below) on $\ssS$ via the formula
\begin{equation}\label{eq:current-bilinear-corresp}
	\eta(\kappa(\epsilon,\epsilon'),v) = B(\epsilon,v\cdot\epsilon')
\end{equation}
for all $\epsilon,\epsilon'\in\ssS$, $v\in V$.

\begin{Definition}[admissible bilinear]\label{def:admissible-bilinear}
	Let $(V,\eta)$ be a (pseudo-)inner product space and let $\ssS$ be a real irreducible module of $\Cl(V)$ (a real pinor module). A real bilinear form $B$ on $\ssS$ is \emph{admissible} if
	\begin{enumerate}\itemsep=0pt
		\item $B$ is either symmetric or skew-symmetric
	$B(\epsilon,\epsilon') = \varsigma_B B(\epsilon',\epsilon)
$
			for all $\epsilon,\epsilon'\in\ssS$, where $\varsigma_B=\pm 1$ is called the \emph{symmetry} of $B$.
		\item Clifford multiplication by an element $v\in V$ is either $B$-symmetric or $B$-skew-symmetric
$
				B(\epsilon,v\cdot \epsilon') = \tau_B B(v\cdot\epsilon,\epsilon')
$
			for all $\epsilon,\epsilon'\in\ssS$, where $\tau_B=\pm 1$ is called the \emph{type} of $B$.
		\item whenever $\ssS$ is reducible as an $\fso(V)$-module, $\ssS=\ssS_+\oplus \ssS_-$, the submodules $\ssS_\pm$ are either mutually $B$-orthogonal ($B(\ssS_+,\ssS_-)=0$) or $B$-isotropic ($B(\ssS_\pm,\ssS_\pm)=0$); we define the \emph{isotropy} $\iota_B$ of $B$ to be $+1$ in the first case and $-1$ in the second.
	\end{enumerate}
	Furthermore, if $B$ is an admissible bilinear on $\ssS$, the Dirac current $\kappa\colon\Odot^2\ssS\to V$ defined by equation~\eqref{eq:current-bilinear-corresp} satisfies
$
		\kappa(\epsilon,\epsilon') = \varsigma_\kappa\kappa(\epsilon',\epsilon)
$
	for all $\epsilon,\epsilon'\in\ssS$, where $\varsigma_\kappa=\varsigma_B\tau_B=\pm 1$ is called the \emph{symmetry} of $\kappa$.
\end{Definition}

\subsubsection[Signature (1,1)]{Signature $\boldsymbol{(1,1)}$}

Here we choose the representation $\Cl(1,1)\to \RR(2)$ defined by
$\Gamma_0 = {\rm i}\sigma_2 = \Omega$,
$ \Gamma_1 = \sigma_1$,
$ \Gamma_{\mu\nu} = \varepsilon_{\mu\nu}\sigma_3$.
The even subalgebra $\Cl_{\bar 0}(1,1)$ is embedded as the subalgebra of diagonal matrices and is isomorphic as an $\RR$-algebra to $\RR^2$. The pinor module is $\ssS=\RR^2$, which decomposes as a~representation of the even subalgebra as $\ssS=\ssS_+\oplus \ssS_-$, where the irreducible real spinor modules $\ssS_+=\RR\be_1$, $\ssS_-=\RR\be_2$ (where $\be_i$ are the standard basis vectors) are the $\pm 1$ eigenspaces of $\Gamma_{01}=\sigma_3$.

From the classification in \cite{Alekseevsky1997}, we expect to find two independent admissible bilinears, both of which have symmetric Dirac currents with respect to which $\ssS_+$, $\ssS_-$ are mutually orthogonal. It is simple to verify that the bilinear products represented by the matrices $\sigma_1$ and $\Omega=i\sigma_2$ are admissible, with symmetry and isotropy properties given in Table~\ref{table:(1,1)-admissible-bilinears}. Either Dirac current restricts non-trivially to $\ssS_\pm$ (in fact they restrict to the same map $\Odot^2\ssS_\pm\to V$) but neither bilinear does.

\begin{table}[t]
\centering
	\caption{Properties of admissible bilinears and their Dirac currents in signature $(1,1)$.}
	\label{table:(1,1)-admissible-bilinears}
	\setlength{\extrarowheight}{.75ex}
	$\begin{array}{c|ccccc}
		B(\epsilon,\epsilon')	& \varsigma_B & \tau_B	& \iota_B & \varsigma_\kappa & \iota_\kappa	\\
		\hline
		\epsilon^{\mathsf{T}}\sigma_1\epsilon'	& +	& + & - & + & +	\\
		\epsilon^{\mathsf{T}}\Omega\epsilon'	& -	& - & - & + & +
	\end{array}$
\end{table}

\subsubsection[Signature (0,2)]{Signature $\boldsymbol{(0,2)}$}

We define the representation $\Cl_{\bar 0}(0,2)\to \RR(2)$ by
$\Gamma_1 = \sigma_3$,
$\Gamma_2 = \sigma_1$,
$\Gamma_{\mu\nu} = \varepsilon_{\mu\nu}\Omega$.
The even subalgebra $\Cl_{\bar 0}(0,2)$ is embedded as the span of $\1$ and $\Omega$, which is isomorphic to $\CC$ as an $\RR$-algebra. The pinor module is again $\ssS=\RR^2$ but is irreducible under the action of $\Cl_{\overline 0}(0,2)$, so we have a unique irreducible real spinor module $\ssS_1=\ssS$.

The standard inner product and symplectic form on $\ssS$ are both admissible bilinears; since we expect two such independent bilinears, both with symmetric Dirac current, from the classification~\cite{Alekseevsky1997}, this exhausts the possibilities. Table~\ref{table:(0,2)-admissible-bilinears} lists their properties. There are no isotropy signs because $\ssS$ is irreducible under the action of $\Cl_{\overline 0}(0,2)$.

\begin{table}[t]
\centering
	\caption{Properties of admissible bilinears and their Dirac currents in signature $(0,2)$.}
	\label{table:(0,2)-admissible-bilinears}
	\setlength{\extrarowheight}{.75ex}
	$\begin{array}{c|ccc}
		B(\epsilon,\epsilon')	& \varsigma_B	& \tau_B & \varsigma_\kappa	\\
		\hline
		\epsilon^{\mathsf{T}}\epsilon'	& +	& + & +	\\
		\epsilon^{\mathsf{T}}\Omega\epsilon'	& -	& - & +
	\end{array}$
\end{table}

\subsection{Adjoints and Fierz identity}

The \emph{Fierz identity} is a formula which allows us to rearrange products of (s)pinors involving bilinears. It is simplest to express using conjugate spinor notation; for $\epsilon\in \ssS$, we denote by $\overline\epsilon$ the element of $\ssS^*$ given by $\epsilon'\mapsto\overline\epsilon\epsilon':=B(\epsilon,\epsilon')$ (where we have fixed the bilinear $B$). Then as well as the inner product $\overline\epsilon\epsilon'$, we have an \emph{outer product} (or tensor product) $\epsilon\overline\epsilon':=\epsilon\otimes\overline\epsilon'\in \ssS\otimes \ssS^*=\End \ssS$. If we choose an abstract basis for $\ssS$ with indices $a,b,\ldots=1,2$ in which $\epsilon$ is represented by the vector $(\epsilon^a)_{a=1,2}$, its adjoint $\overline\epsilon$ by the covector $(\overline\epsilon_a)_{a=1,2}$, a bilinear $B$ by the matrix $(B_{ab})_{a,b=1,2}$, and an endomorphism $\Gamma$ by the matrix $(\Gamma\indices{^a_b})_{a,b=1,2}$, we can write the following in Einstein notation:
$
	\overline{\epsilon}_b = \epsilon^a B_{ab}$,
	$\overline{\epsilon}\epsilon = \overline{\epsilon}_b\epsilon^b = B_{ab}\epsilon^a\epsilon^b$,
	$(\epsilon\overline{\epsilon})\indices{^a_b} = \epsilon^a\overline{\epsilon}_b = \epsilon^a \epsilon^c B_{cb}$.

\begin{Proposition}[Fierz identity] \label{prop:2d-fierz}
	Let $\ssS$ be the real irreducible pinor module in signature $(1,1)$ or $(0,2)$ and $B$ either of the admissible bilinears from the appropriate table above. Then for $\epsilon,\epsilon'\in \ssS$, we have
\[
		\epsilon\overline\epsilon' = \frac{1}{2}\qty((\overline\epsilon'\epsilon)\1 + (\overline\epsilon'\Gamma_\mu\epsilon)\Gamma^\mu - \varsigma(\overline\epsilon'\Gamma_*\epsilon)\Gamma_*).
\]
\end{Proposition}

\begin{proof}
	Since $\epsilon\overline\epsilon'$ is an endomorphism of $S$, it is an $R$-linear combination of $\Gamma$-matrices; we have~${
		\epsilon\overline\epsilon' = a\1 + b_\mu\Gamma^\mu + c\Gamma_*}
$
	for some $a,b_\mu,c\in\RR$. We see from the explicit matrix representation given in Section~\ref{sec:2d-dirac-currents} that $\Gamma_\mu$ and $\Gamma_*$ are traceless, while $\tr\1=\dim \ssS = 2$, whence we compute~${\tr(\epsilon\overline\epsilon') = 2a}$,	
$\tr(\Gamma_\nu\epsilon\overline\epsilon') = 2b_\nu$, $
		\tr(\Gamma_*\epsilon\overline\epsilon') = -\varsigma2c$.
	On the other hand, for any $\Gamma\in\End \ssS$, $\tr(\Gamma\epsilon\overline\epsilon')=\overline\epsilon'\Gamma\epsilon$, which can be verified in the explicit representation or by noting that in Einstein notation, both expressions are equal to $B_{ab}\Gamma\indices{^b_c}\epsilon'^a\epsilon^c$. The result follows immediately.
\end{proof}

We can now render the definition \eqref{eq:current-bilinear-corresp} of the Dirac current $\kappa\colon\Odot^2\ssS\to V$ in the more convenient form
$\kappa(\epsilon,\epsilon')^\mu = B(\epsilon,\Gamma^\mu\epsilon')$.
Since $\kappa$ is necessarily symmetric in either signature, for~$\epsilon\in \ssS$ we define $\kappa_\epsilon:=\kappa(\epsilon,\epsilon)\in V$ which has components
$
	\kappa_\epsilon^\mu=\overline\epsilon\Gamma^\mu\epsilon=B(\epsilon,\Gamma^\mu\epsilon)$.

The Fierz identity has consequences for the causal properties of the Dirac current.

\begin{Corollary}\label{coro:2d-dirac-causality}
	Let $\epsilon\in \ssS$. Then
$
		\norm{\kappa_\epsilon}^2 = (\overline\epsilon\epsilon)^2 + \varsigma(\overline\epsilon\Gamma_*\epsilon)^2$.
	In particular, in the Riemannian case $(\varsigma=+1)$, $\kappa_\epsilon=0$ if and only if $\epsilon=0$; in the Lorentzian case $(\varsigma=-1)$, $\kappa$ is null if and only if $\epsilon$ is chiral, otherwise $\kappa$ is spacelike for $\varsigma_B=+1$ and timelike for $\varsigma_B=-1$.
\end{Corollary}

\begin{proof}
	For the first claim, applying the Fierz identity to the expression $(\overline\epsilon\epsilon)^2$ gives us
\[
		(\overline\epsilon\epsilon)^2 = \overline\epsilon(\epsilon\overline\epsilon)\epsilon
			= \frac{1}{2}\qty((\overline\epsilon\epsilon)(\overline\epsilon\epsilon) + (\overline\epsilon\Gamma_\mu\epsilon)(\overline\epsilon\Gamma^\mu\epsilon) - \varsigma(\overline\epsilon\Gamma_*\epsilon)(\overline\epsilon\Gamma_*\epsilon)),
\]
	which can be rearranged to give the desired expression. Now, consider the two quadratic forms~$\overline\epsilon\Gamma_*\epsilon$ and $\overline\epsilon\epsilon$ and note that $\overline\epsilon\Gamma_*\epsilon=0$ if $\varsigma_B=+1$ and $\overline\epsilon\epsilon=0$ if $\varsigma_B=-1$. For $\varsigma=+1$, the expression which does not vanish identically vanishes if and only if $\epsilon=0$ (check this in the explicit representation). For $\varsigma=-1$, since the isotropy $\iota_B=-1$ for either admissible bilinear (see Table~\ref{table:(1,1)-admissible-bilinears}) and $\Gamma_*$ preserves the chiral subspaces, the expression which does not vanish identically vanishes if and only if $\epsilon\in \ssS_\pm$. This proves the second claim.
\end{proof}

\section{Spencer cohomology and filtered subdeformations}
\label{sec:2d-spencer}

We now consider flat model superalgebras $\fs$ (as defined in \cite[Definition~2.1]{Beckett2024_2}) associated to the inner product space $V=\RR^{1,1}$ or $\RR^{0,2}$ with odd part $S=\ssS$ or $S=\ssS_+$ (the latter only in the Lorentzian case) and the Dirac current $\kappa\colon\Odot^2S\to V$ being the one associated to one of the two admissible bilinears described in Table~\ref{table:(1,1)-admissible-bilinears} or Table~\ref{table:(0,2)-admissible-bilinears} respectively, or a restriction thereof (more details below). Note that only superalgebras are possible here since all Dirac currents are symmetric. In either signature, the two Dirac currents on $\ssS$ can be distinguished by the invariant $\varsigma_B=\pm 1$. We will determine the Spencer cohomology group $\ssH^{2,2}(\fs_-;\fs)$ for each such superalgebra as well as their maximally supersymmetric filtered subdeformations.

We recall that each superalgebra $\fs$ is $\ZZ$-graded with $\fs_{-2}=V$, $\fs_{-1}=S$, $\fs_0=\fso(V)$ and $\fs_i=0$ otherwise, and its brackets are
\begin{gather}
	\comm{A}{B} = AB-BA = 0,\qquad	
	\comm{A}{v} = Av,\qquad
	 \comm{v}{w} = 0,\nonumber	\\
	\comm{A}{s} = A\cdot s = \frac{1}{2}\omega_A\cdot s,\qquad
	 \comm{v}{s} = 0,\qquad
	 \comm{s}{s'} = \kappa(s,s'),\label{eq:2d-flat-brackets}
\end{gather}
for $A,B\in\fs_0$, $s,s'\in\fs_{-1}$, $v,w\in\fs_{-2}$, and we will make use of depolarised equations involving symmetric combinations of spinors; for example the odd-odd bracket can be defined as $\comm{s}{s}=\kappa_s:=\kappa(s,s)$ for all $s\in S$. Note that since $\dim V=2$, $\dim\fs_0=\dim\fso(V)=1$, hence the first bracket must vanish.

In signature $(0,2)$, the pinor representation is irreducible under the action of the spin group, thus we take $S=\ssS_1=\ssS$ and for either choice of Dirac current from Table~\ref{table:(0,2)-admissible-bilinears}, $\fs$ is a minimal flat model superalgebra. We note that $\End S\cong\Cl(V)$.

In the $(1,1)$ case, the pinor representation is reducible under the spin group with $\ssS=\ssS_+\oplus\ssS_-$. Table~\ref{table:(1,1)-admissible-bilinears} gives the two possible Dirac currents on $\ssS$, both of which have $\iota_\kappa=+1$ and restrict to the same non-trivial Dirac current on $\ssS_\pm$. Thus for either map we can choose $S=\ssS$ (which we will call the ``non-chiral'' case) or, without loss of generality, $S=\ssS_+$ (which we call ``chiral''). In the non-chiral case we have $\End S\cong\Cl(V)$, while in the chiral case $\End S\cong\Cl_{\overline 0}(V)$. In what follows, we will work mainly with the non-chiral case and then bootstrap the results for the chiral case from it at the end. We note that only the chiral case gives a \emph{minimal} flat model superalgebra.

\subsection{Spencer cohomology}

In either signature, let us fix $S=\ssS$ and an admissible bilinear $B$ and corresponding Dirac current $\kappa$, recalling that in either signature there are two choices for $B$ (or $\kappa$) distinguished by $\varsigma_B=\pm 1$. We let $\fs$ be the corresponding flat model. We showed in \cite[Theorem~3.10]{Beckett2024_2} that the Killing superalgebra associated to an admissible connection (see Definition~\ref{def:killing-spinor} and also~\cite[Definition~3.6]{Beckett2024_2}) on the spinor bundle of a spin manifold with Dirac current $\kappa$ is a filtered subdeformation of a flat model superalgebra $\fs$; that is, it is a filtered Lie superalgebra whose associated graded Lie superalgebra is isomorphic to a graded subalgebra $\fa$ of $\fs$.

We recall that the Spencer complex $(\ssC^{\bullet,\bullet}(\fs_-;\fs),\partial)$ of $\fs$ is the Chevalley--Eilenberg complex of the graded subalgebra $\fs_-=\Otimes_{i<0}\fs_i=V\oplus S$ with values in the module $\fs$, where the action is the restriction of the adjoint representation of $\fs$ to $\fs_-$, and apart from the homological grading, it inherits a grading from that of $\fs$ compatible with the differential $\partial$. In particular, the graded components of the cochain spaces are $\ssC^{d,p}(\fs_-;\fs)=0$ for $p<0$ and
\[
	\ssC^{d,p}(\fs_-;\fs) = \qty(\Wedge^p \fs_-^*\otimes \fs)_d = \Hom\qty(\Wedge^p \fs_-,\fs)_d
\]
for $p\geq 0$, where $\Wedge^p$ is taken in the ``super-sense'' and the subscript $d$ indicates maps of degree~$d$ with respect to the grading inherited from $\fs$. We will not explicitly give the formula for $\partial$ here; for a full description see \cite[Section~2.2]{Beckett2024_2} or the original work of Cheng and Kac on filtered deformations of graded superalgebras \cite{Cheng1998}. The natural action of $\fs$ on the full cochain spaces~$\ssC^{\bullet,p}(\fs_-;\fs) = \Wedge^p \fs_-^*\otimes \fs = \Hom\qty(\Wedge^p \fs_-,\fs)$ restricts to an action of the subalgebra~$\fs_0=\fso(V)$ which preserves the grading $d$ and differential $\partial$, so it also preserves the spaces of graded cocycles~$\ssZ^{d,p}(\fs_-;\fs)$ and coboundaries $\ssB^{d,p}(\fs_-;\fs)$, whence there is an induced action of $\fs_0$ on the graded cohomology groups $\ssH^{d,p}(\fs_-;\fs)=\ssZ^{d,p}(\fs_-;\fs)/\ssB^{d,p}(\fs_-;\fs)$.

Of most relevance here is the cohomology group $\ssH^{2,2}(\fs_-;\fs)$ which contains the infinitesimal filtered deformations of certain graded subalgebras (namely the maximally supersymmetric ones) of $\fs$, about which we will say more in Section~\ref{sec:2d-max-susy-subdef}. The $d=2$ subcomplex is
\begin{align*}
	0 &\longrightarrow \ssC^{2,1}(\fs_-;\fs) = \Hom(V,\fso(V)) \\
	&\longrightarrow \ssC^{2,2}(\fs_-;\fs) = \Hom\bigl(\Wedge^2V,V\bigr)\oplus\Hom(V\otimes S,S)\oplus\Hom\bigl(\Odot^2S,\fso(V)\bigr)\\
	&\longrightarrow \ssC^{2,3}(\fs_-;\fs) = \Hom\bigl(V\otimes\Odot^2S,V\bigr)\oplus\Hom\bigl(\Odot^3S,S\bigr)\\
	&\longrightarrow 0.
\end{align*}
The first non-trivial codifferential $\partial\colon \ssC^{2,1}(\fs_-;\fs) \to \ssC^{2,2}(\fs_-;\fs)$ is injective and projects to an isomorphism onto the first component of $\ssC^{2,2}(\fs_-;\fs)$, whence any $(2,2)$-cocycle (i.e., an element of $\ssZ^{2,2}(\fs_-;\fs)$) is homologous to a unique cocycle $\beta+\gamma$ with $\beta\in\Hom(V\otimes S,S)$ and $\gamma\in\Hom\bigl(\Odot^2S,\fso(V)\bigr)$. We call a cocycle of the latter form \emph{normalised} and denote the space of normalised cocycles by $\cH^{2,2}$; it follows that $\cH^{2,2}\cong\ssH^{2,2}(\fs_-;\fs)$ as $\fs_0=\fso(V)$-modules \cite[Lemma~4.4]{Beckett2024_2}. Explicitly, the cocycle condition $\partial(\beta+\gamma)=0$ is equivalent to
\begin{align}
	&2\kappa(s,\beta(v,s)) + \gamma(s,s)v = 0, \label{eq:spencer-1}\\
	&\beta(\kappa_s,s) + \gamma(s,s)\cdot s = 0 \label{eq:spencer-2}
\end{align}
for all $s\in S$, $v\in V$ -- these equations also appeared as \cite[equations~(4.2) and (4.3)]{Beckett2024_2}. As in loc.\ cit., we refer to them as the \emph{normalised Spencer cocycle conditions} for $(2,2)$-cochains since $\cH^{2,2}$ is precisely the space of solutions to this system.

Let us now solve this system of equations using a method previously used in work applying Spencer cohomology to supergravity \cite{Beckett2021,deMedeiros2016,deMedeiros2018,Figueroa-OFarrill2017}. We begin by parametrising $\beta$. Since $\Hom(V\otimes S,S)\cong\Hom(V,\End S)$, for $v\in V$ we can write $\beta_v\in\End S$ for the endomorphism $s\mapsto\beta_v(s):=\beta(v,s)$, and $\beta_\mu:=\beta_{e_\mu}$ in the orthonormal basis $\{e_\mu\}$. Then since $\End S\cong \Cl(V)\cong\RR(2)$, $\beta$ can be parametrised as
$\beta_\mu = a_\mu \1 + b_{\mu\nu}\Gamma^\nu + c_\mu \Gamma_*
$
where we use the Einstein summation convention and the coefficients take real values; we can consider them to be the components of some $a,c\in V^*$, $b\in \Otimes^2V^*$. Now, the first cocycle condition \eqref{eq:spencer-1} is equivalent to
\[
	\gamma(s,s)_{\mu\nu} = -2\overline{s}\Gamma_\mu\beta_\nu s,
\]
where in forming the conjugate $\overline{s}$ we use the admissible bilinear $B$. This equation completely determines $\gamma$ in terms of $\beta$, and since $\gamma$ must take values in $\fso(V)$, it places the following constraint on $\beta$:
\begin{equation}\label{eq:cocycle-1-constr}
	\overline{s}\Gamma_{(\mu}\beta_{\nu)} s = 0
\end{equation}
for all $s\in S$. Using our parametrisation for $\beta$ and evaluating products of $\Gamma$-matrices gives us
\begin{align*}
	\overline{s}\Gamma_\mu\beta_\nu s
		&= a_\nu\overline{s}\Gamma_\mu s + b_{\nu\rho}\overline{s}\Gamma_\mu\Gamma^\rho s + c_\nu\overline{s}\Gamma_\mu\Gamma_*s	\\
		&= b_{\nu\mu}\overline{s}s + (a_\nu\eta_{\mu\rho} + \varsigma c_\nu\varepsilon_{\mu\rho})\overline{s}\Gamma^\rho s + \varepsilon\indices{_\mu^\rho} b_{\nu\rho}\overline{s}\Gamma_* s,
\end{align*}
where we recall that $\varepsilon_{\mu\nu}$ here denotes the Levi-Civita symbol. If $\varsigma_B=+1$, note that $\overline s\Gamma_*s=0$ for all $s\in S$, and one can show that \eqref{eq:cocycle-1-constr} holds if and only if
$b_{(\nu\mu)} = 0
$ and $
	a_{(\nu}\eta_{\mu)\rho}+\varsigma c_{(\nu}\varepsilon_{\mu)\rho} = 0$ for $\varsigma_B=+1$.
Since we are working in two dimensions, the first equation implies that ${b_{\mu\nu}=b\varepsilon_{\mu\nu}}$ for some $b\in\RR$. Fully symmetrising the latter equation gives $a_{(\mu}\eta_{\nu\rho)}=0$, and tracing this gives ${a=0}$. Substituting this back into the full equation then gives us $c=0$. On the other hand, if $\varsigma_B=-1$, $\overline{s}s = 0$ for all $s\in S$, and \eqref{eq:cocycle-1-constr} holds if and only if
\smash{$
	\varepsilon\indices{_{(\mu}^\rho} b\indices{_{\nu)\rho}} = 0
$} and $
	a_{(\nu}\eta_{\mu)\rho}+\varsigma c_{(\nu}\varepsilon_{\mu)\rho} = 0$
 for $ \varsigma_B=-1$.
One can show (for instance by substituting values for $\mu$, $\nu$) that the first equation is satisfied if and only if $b_{\mu\nu}=b\eta_{\mu\nu}$ and again the latter equation has only the trivial solution. Thus we have shown that $\beta$ is parametrised by a single parameter $b\in \RR$, with
\begin{equation}\label{eq:2d-real-beta-soln}
	\beta_\mu =
	\begin{cases}
		b\varepsilon_{\mu\nu}\Gamma^\nu	& \text{for } \varsigma_B=+1,\\
		b\Gamma_\mu						& \text{for } \varsigma_B=-1.
	\end{cases}
\end{equation}
Substituting this back into our equation for $\gamma$, we have
\begin{equation}\label{eq:2d-real-gamma-soln}
	\gamma(s,s)_{\mu\nu} =
	\begin{cases}
		2b\varepsilon_{\mu\nu}\overline{s}s				& \text{for } \varsigma_B=+1,\\
		-2b\varepsilon_{\mu\nu}\overline{s}\Gamma_*s	& \text{for } \varsigma_B=-1.
	\end{cases}
\end{equation}
The remaining cocycle condition \eqref{eq:spencer-2} is then identically satisfied by the Fierz identity, as we will now demonstrate. We have
\begin{align*}
	\beta(\kappa_s,s) + \gamma(s,s)\cdot s
		&= \overline{s}\Gamma^\mu s \beta_\mu s + \frac{1}{4}\gamma(s,s)_{\mu\nu}\Gamma^{\mu\nu}s	\\
		&=
			\begin{cases}
				b\varepsilon_{\mu\nu}(\overline{s}\Gamma^\mu s)\Gamma^\nu s + \varsigma b(\overline{s}s)\Gamma_* s	& \text{for } \varsigma_B=+1, 	\\
				b(\overline{s}\Gamma_\mu s)\Gamma^\mu s - \varsigma b(\overline{s}\Gamma_* s)\Gamma_* s & \text{for } \varsigma_B=-1.
			\end{cases}
\end{align*}
Now, the Fierz identity with either bilinear (see Proposition~\ref{prop:2d-fierz}) gives us
\[
	s\overline{s} = \frac{1}{2}\qty((\overline{s}s)\1 + (\overline{s}\Gamma^\mu s)\Gamma_\mu - \varsigma(\overline{s}\Gamma_* s) \Gamma_*),
\]
but we note that the first term on the right-hand side vanishes for $\varsigma_B=-1$, and the third vanishes for $\varsigma_B=+1$. Thus for $\varsigma_B=+1$,
\begin{align*}
	\varepsilon_{\mu\nu}(\overline{s}\Gamma^\mu s)\Gamma^\nu s
		&= \varepsilon_{\mu\nu}\Gamma^\nu (s\overline{s})\Gamma^\mu s	= \frac{1}{2}(\overline{s}s)\varepsilon_{\mu\nu}\Gamma^\nu\Gamma^\mu s
			+ \frac{1}{2}(\overline{s}\Gamma_\rho s)\varepsilon_{\mu\nu}\Gamma^\nu\Gamma^\rho\Gamma^\mu s	= -\varsigma(\overline{s}s)\Gamma_*s,
\end{align*}
and for $\varsigma_B=-1$,
\begin{align*}
	(\overline{s}\Gamma_\mu s)\Gamma^\mu s
		&= \Gamma^\mu (s\overline{s})\Gamma_\mu s	= \frac{1}{2}(\overline{s}\Gamma^\rho s)\Gamma^\mu\Gamma^\rho\Gamma_\mu s - \varsigma\frac{1}{2}(\overline{s}\Gamma_* s)\Gamma^\mu\Gamma_*\Gamma_\mu s	= \varsigma(\overline{s}\Gamma_* s)\Gamma_* s,
\end{align*}
whence $\beta(\kappa_s,s) + \gamma(s,s)\cdot s=0$, so we have solved the Spencer cocycle conditions \eqref{eq:spencer-1} and~\eqref{eq:spencer-2} in the minimal case for $(0,2)$ and the minimal non-chiral case for $(1,1)$ ($S=\ssS$). In the minimal chiral $(1,1)$ case ($S=\ssS_+$), we can follow a similar argument except that we must set $b_{\mu\nu}=c_\mu=0$ in the parametrisation of $\beta$, whence we have only the trivial solution. In particular, we have shown the following.

\begin{Proposition}\label{prop:2d-real-spencer-soln}
	In signature $(1,1)$ or $(0,2)$, let $\fs$ be the flat model superalgebra defined by ${S=\ssS}$ and $\kappa$ any of the Dirac currents from Table {\rm\ref{table:(1,1)-admissible-bilinears}} or Table~{\rm\ref{table:(0,2)-admissible-bilinears}}, respectively. Then
${\ssH^{2,2}(\fs_-;\fs) \cong \RR}$
	as an $\fs_0=\fso(V)$-module. The space of normalised cocycles $\cH^{2,2}$ consists of elements $\beta+\gamma$ given by \eqref{eq:2d-real-beta-soln}, \eqref{eq:2d-real-gamma-soln} for $b\in\RR$.
	
In signature $(1,1)$, if $S=\ssS_+$, then there is a unique non-trivial Dirac current $($up to rescaling$)$ and we have
$\ssH^{2,2}(\fs_-;\fs) = 0$.
\end{Proposition}

In either signature, if $S=\ssS$, one could define the Dirac current $\kappa$ to be any linear combination of those from Table \ref{table:(1,1)-admissible-bilinears} or Table~\ref{table:(0,2)-admissible-bilinears} (as appropriate). One could then view $\kappa$ as being obtained from a~bilinear $B$ which is not admissible unless it is a non-zero scalar multiple of one of the two from the table. One can show that $\ssH^{2,2}(\fs_-;\fs)=\RR$ whenever $B$ is non-degenerate. However, since the general calculation is somewhat more involved and the case where $B$ is not admissible is perhaps of limited interest, we omit the details.

\subsection{Maximally supersymmetric filtered subdeformations}
\label{sec:2d-max-susy-subdef}

Let us now use the result above to describe the filtered subdeformations of $\fs$ with odd part ${\fs_{-1}=S}$; i.e., the maximally supersymmetric deformations, in the terminology of \cite[Definition~4.10]{Beckett2024_2}. Since we found that the Spencer cohomology in the chiral case ($S=\ssS_+$) in Lorentzian signature is trivial, we consider only the non-chiral case ($S=\ssS$) in either signature. Note that there is no sub-maximal highly supersymmetric case here since $\dim S=2$.

In \cite[Section~4.5.5]{Beckett2024_2}, it is argued that odd-generated maximally supersymmetric filtered subdeformations are determined by normalised cocycles $\beta+\gamma\in\cH^{2,2}$ which are invariant under the action of the subalgebra $\fh_{(\beta+\gamma)}:=\gamma(\fD)$ of $\fso(V)$, where $\fD=\ker\kappa\subseteq\Odot^2 S$ is the Dirac kernel, and which satisfy some integrability conditions. In the present case, $\dim\fso(V)=1$, so we must either have $\fh_{(\beta+\gamma)}=0$ or $\fh_{(\beta+\gamma)}=\fso(V)$. Moreover, any element of $\cH^{2,2}$ is actually $\fso(V)$-invariant since $\cH^{2,2}\cong\RR$, the trivial $\fso(V)$-module. In particular, it suffices to study filtered deformations $\fstilde$ of the whole superalgebra $\fs$, since any deformation of the unique proper maximally supersymmetric subalgebra $\fs_-$ can be extended to a deformation of $\fs$.

It remains only to check the integrability conditions, which were developed in \cite[Section~4.4.2]{Beckett2024_2}. We adapt the definition of integrable cocycles \cite[Definition~4.23]{Beckett2024_2} as follows. Fixing~${\beta+\gamma\in\cH^{2,2}}$, let $\Theta\colon V\otimes\Odot^2S\to\fso(V) $ be the map defined by
$\Theta(v,s,s) = 2\gamma\qty(s,\beta(v,s))
$
for~${v\in V}$, $s\in S$. We say that $\beta+\gamma$ is \emph{integrable} if the following hold.
\begin{enumerate}\itemsep=0pt
	\item The map $\Theta$ annihilates the Dirac kernel $\fD$ so factors through a map $\theta\colon V\otimes V\to\fso(V)$ making the following diagram commute
	\[
	\begin{tikzcd}
		V\otimes \Odot^2 S' \ar[rr,"\Theta"]\ar[rd,"\Id\otimes \kappa"'] & & \fso(V)\\
			& V \otimes V \ar[ru,"\theta"']. &
	\end{tikzcd}
	\]
	It follows that this map $\theta$ is skew-symmetric.
	\item The map $\theta\colon \Wedge^2 V\to \fso(V)$ satisfies
	\[
		\theta(v,w)\cdot s
		= \beta\qty(v,\beta(w,s)) - \beta\qty(w,\beta(v,s))
	\]
	for all $v,w\in V$, $s\in S$.
If $\beta+\gamma$ is integrable (and non-zero) then it defines a (non-trivial) filtered deformation $\fstilde$ of $\fs$ with the following brackets (compare with undeformed brackets~\eqref{eq:2d-flat-brackets} on the same underlying vector space)
\begin{gather}
		\comm{A}{B} = AB-BA=0,\qquad
		 \comm{A}{v} = Av \in S,\qquad
		\comm{v}{w} = \theta(v,w) \in\fso(V),\nonumber	\\
		\comm{A}{s} = A\cdot s=\frac{1}{2}\omega_A\cdot s \in S,\qquad
		 \comm{v}{s} = \beta(v,s) \in S,\nonumber\\
		 \comm{s}{s}= \kappa_s + \gamma(s,s) \in V\oplus\fso(V),\label{eq:2d-deformed-brackets}
\end{gather}
for $A,B\in\fs_0$, $s,s'\in\fs_{-1}$, $v,w\in\fs_{-2}$, and every filtered deformation of $\fs$ is obtained in this manner.
	\end{enumerate}

We will show that every element of $\cH^{2,2}$ is in fact integrable and then describe the filtered deformations more explicitly. For the first condition, a simple computation in an orthonormal basis shows that
\[
	\Theta_{\mu\nu\rho}(s,s)
	:= \Theta_{\mu\nu}(e_\rho,s,s)
	= 2\gamma_{\mu\nu}(s,\beta_\rho s)
	= \theta_{\mu\nu\rho\sigma}\kappa_s^\sigma,
\]
where
\[
	\theta_{\mu\nu\rho\sigma} =
		\begin{cases}
			4b^2\varepsilon_{\mu\nu}\varepsilon_{\rho\sigma}
				& \text{for } \varsigma_B = +1,	\\
			\sigma 4b^2\varepsilon_{\mu\nu}\varepsilon_{\rho\sigma}
				& \text{for } \varsigma_B = -1.
		\end{cases}
\]
Thus we have shown that $\Theta$ factors through a map $\theta\colon\Wedge^2 V\to \fso(V)$ with components given above, hence the first integrability condition is satisfied. In component form, the second integrability condition is
$\frac{1}{4}\theta_{\mu\nu\rho\sigma}\Gamma^{\rho\sigma}s = \comm{\beta_\mu}{\beta_\nu}s
$
for all $s\in S$. The left-hand side can be rewritten~as
\[
	\frac{1}{4}\theta_{\mu\nu\rho\sigma}\Gamma^{\rho\sigma}s
		= \frac{1}{4}\theta_{\mu\nu\rho\sigma}\varepsilon^{\rho\sigma}\Gamma_*s
		= 	\begin{cases}
				\varsigma 2b^2\varepsilon_{\mu\nu}\Gamma_*
					& \text{for } \varsigma_B = +1,	\\
				2b^2\varepsilon_{\mu\nu}\Gamma_*
					& \text{for } \varsigma_B = -1.
			\end{cases}
\]
The commutator in the right-hand side can easily be computed to give
\[
	\comm{\beta_\mu}{\beta_\nu} =
		\begin{cases}
			\varsigma 2b^2\varepsilon_{\mu\nu}\Gamma_*
				& \text{for } \varsigma_B = +1,	\\
			2b^2\varepsilon_{\mu\nu}\Gamma_*
				& \text{for } \varsigma_B = -1,
		\end{cases}
\]
whence the integrability condition is identically satisfied.

Thus for either signature and any choice of Dirac current, there is a one-parameter family of filtered deformations of $\fs$ with parameter $b\in\RR$. Let us describe the brackets \eqref{eq:2d-deformed-brackets} of such a~deformation $\fstilde$ for fixed $b\neq 0$ in terms of an explicit basis for the even part of the superalgebra $\fs$. In our chosen orthonormal basis, let $P_\mu$ denote the infinitesimal translation generators (of which there are two in either signature) and let $L_{\mu\nu}$ denote the infinitesimal generators of $\fso(V)$; in our case there is one such generator $L_{\mu\nu}=\varepsilon_{\mu\nu} L_*$ where $L_* = L_{01}$ in signature $(1,1)$ and $L_*=L_{12}$ in signature $(0,2)$.

The brackets $\comm{\fs_0}{\fs_0}$, $\comm{\fs_0}{\fs_{-1}}$ and $\comm{\fs_0}{\fs_{-2}}$ are not deformed; the first is trivial while the others are
$\comm{L_*}{s} = \frac{1}{2}\Gamma_* s$,
$\comm{L_*}{P_\mu} = -\varsigma \varepsilon_{\mu\nu}P^\nu$,
and the deformed brackets take the following form:
\[
	\comm{P_\mu}{P_\nu} = \varsigma 4b^2\varepsilon_{\mu\nu}L_*,	\qquad
	\comm{P_\mu}{s} = b\varepsilon_{\mu\nu}\Gamma^\nu s,	\qquad
	\comm{s}{s} = \kappa_s^\mu P_\mu + \varsigma 2b(\overline s s) L_*,	\qquad \varsigma_B=+1,
\]
or
\[
	\comm{P_\mu}{P_\nu} = 4b^2\varepsilon_{\mu\nu}L_*,	\qquad
	\comm{P_\mu}{s} = b\Gamma_\mu s,	\qquad
	\comm{s}{s} = \kappa_s^\mu P_\mu - \varsigma 2b\qty(\overline s\Gamma_*s) L_*	\qquad \varsigma_B=-1.
\]
Note that since $\comm{V}{V}=\RR L_*=\fstilde_0$ if $b\neq 0$, there are no maximally supersymmetric proper subalgebras; the only non-trivial maximally supersymmetric filtered subdeformations are deformations of the whole graded superalgebra $\fs$.

In each case, the even part of $\fstilde$ is the isometry algebra (algebra of Killing vectors) for a maximally symmetric pseudo-Riemannian geometry; we have the isometry algebra of 2-dimensional hyperbolic space with scalar curvature $R=-8b^2$ for either sign $\varsigma_B$ in the Riemannian case, and~$dS_2$ $\bigl(R=8b^2\bigr)$ for $\varsigma_B=+1$ and~$AdS_2$ $\bigl(R=-8b^2\bigr)$ for $\varsigma_B=-1$ in the Lorentzian case. Indeed, each of those geometries is a homogeneous space for the metric Lie pair $\bigl(\fstilde_{\overline 0},\fs_0=\fso(V),\eta\bigr)$, where we note that $V\cong\fstilde_{\overline 0}/\fso(V)$ as an $\fso(V)$-module. Moreover, for $\varsigma=\varsigma_B=-1$, $\fstilde$ is actually the standard anti-de Sitter superalgebra (see \cite{VanProeyen1999}). Upon considering Killing spinors and supersymmetric geometries in 2 dimensions below, we will find that these homogeneous spaces are precisely the maximally supersymmetric geometries, at least up to local isometry.

\section{Admissible connections and Killing superalgebras}
\label{sec:2d-ksa}

Throughout, we let $(M,g)$ be a connected 2-dimensional pseudo-Riemannian spin manifold of either Riemannian\footnote{Recall that we take a $+$ sign in the Clifford relation \eqref{eq:clifford-rel} so that when we work with a Riemannian (positive-definite) metric, the signature is $(0,2)$ for the purposes of Clifford algebra constructions.} or Lorentzian signature; in Lorentzian signature we additionally assume that $(M,g)$ is time-orientable (so that it is strongly spin in the sense of, e.g., \cite{Cortes2021}). Fixing a spin structure $P\to M$, we let \smash{$\Sbundle:=P\times_{\Spin(V)} \ssS$} be the spinor bundle with fibre $\ssS$ and $\fS=\Gamma(\Sbundle)$ be its space of sections. We denote the Levi-Civita connection (and its lift to $\Sbundle$) by $\nabla$. We let $B$ be either of the admissible bilinears on $S$ in each signature; since $B$ is $\fso(V)$-invariant, it induces a $\nabla$-parallel bilinear form $\pair{-}{-}$ on $\Sbundle$. There is then a Dirac current on the spinor bundle $\kappa\colon\Odot^2 \Sbundle\to TM$ defined by the equation
$
	g(\kappa(\epsilon,\epsilon'),X) = \pair{\epsilon}{X\cdot \epsilon'}
$
for all $X\in\fX(M)$, $\epsilon,\epsilon'\in\fS$.

Let $\vol$ denote the canonical volume form on $(M,g)$. Then in a local orthonormal frame, $\vol\cdot\epsilon=\Gamma_*\epsilon$ for all $\epsilon\in\fS$. We note that for $\alpha^{(p)}\in\Omega^p(M)$, we have
\[
	\vol\cdot \alpha^{(0)} = \alpha^{(0)}\vol = *\alpha^{(0)},
	\qquad \vol\cdot\alpha^{(1)} = -*\alpha^{(1)},
	\qquad \vol\cdot\alpha^{(2)} = -*\alpha^{(2)},
\]
where $\cdot$ denotes Clifford multiplication.

\subsection{Admissible connections}

We now consider admissible connections $D$ on the spinor bundle $\Sbundle$ equipped with Dirac current~$\kappa$. We can always write such a connection as $D=\nabla-\beta$ for some unique $\beta\in\Omega^1(M;\Sbundle)$. We recall the following definition \cite[Definition~3.6]{Beckett2024_2}, noting that since we deal with the case of a symmetric Dirac current we have simplified the conditions by polarising them.

\begin{Definition}\label{def:killing-spinor}
	The connection $D=\nabla-\beta$ on $\Sbundle$ is \emph{admissible} if the following hold:
	\begin{enumerate}\itemsep=0pt
		\item The section $\gamma$ of the bundle $\Hom\qty(\Odot^2\Sbundle,\End(TM))$ defined by $\gamma(\epsilon,\epsilon)X := - 2\kappa(\epsilon,\beta(X)\epsilon)$ satisfies $\gamma(\epsilon,\epsilon)\in\fso(M,g)$ for all $\epsilon\in\fS$; 
		\item $\beta(\kappa_\epsilon)\epsilon
				+ \gamma(\epsilon,\epsilon)\cdot\epsilon
				= 0$ for all $\epsilon\in\fS$; 
		\item $\eL_{\kappa_\epsilon}\beta = 0$ for all $\epsilon\in\fS_D$. 
	\end{enumerate}
	If $D$ is admissible, the differential equation $D\epsilon=0$ (equivalently $\nabla\epsilon =\beta\epsilon$) is called the \emph{Killing spinor equation} and
$
		\fS_D = \{\epsilon\in\fS \mid D\epsilon = 0\}
$
	the space of \emph{Killing spinors}. Furthermore,
\[
		\fV_D = \{X\in \fX(M) \mid \eL_Xg =0,\, \eL_X\beta = 0\}
\]
	is the space of \emph{restricted Killing vectors}. The \emph{Killing superalgebra} is the vector superspace $\fK_D$ with $(\fK_D)_{\overline 0}=\fV_D$ and $(\fK_D)_{\overline 1}=\fS_D$ equipped with the Lie superalgebra bracket
$\comm{X}{Y} = \eL_X Y$,
$\comm{X}{\epsilon} = \eL_X \epsilon$,
$\comm{\epsilon}{\epsilon} = \kappa_\epsilon$,
	for $X,Y\in\fV_D$ and $\epsilon\in\fS_D$, where $\eL_X \epsilon=\nabla_X\epsilon - (\nabla X)\cdot\epsilon$ is the spinorial Lie derivative of \cite{Kosmann1971}.
\end{Definition}

The first two conditions in the definition (taken pointwise) are essentially the normalised Spencer cocycle conditions \eqref{eq:spencer-1} and \eqref{eq:spencer-2} in degree $(2,2)$ for the appropriate flat model superalgebra\footnote{Since we concluded that $\cH^{2,2}=0$ for the chiral case $S=\ssS_+$ in signature $(1,1)$, the only admissible connection is $\nabla$ in that case, meaning that Killing spinors are nothing but parallel spinors, hence we only consider the non-chiral case.} $\fs$, the solution to which is given by equations \eqref{eq:2d-real-beta-soln} and \eqref{eq:2d-real-gamma-soln}. In global notation, we have
\begin{equation}\label{eq:2d-real-beta}
	\beta(X) \epsilon = \begin{cases}
		b X\cdot \epsilon	& \text{ for } \varsigma_B = -1,\\
		b (*X)\cdot\epsilon	& \text{ for } \varsigma_B = +1,
	\end{cases}
\end{equation}
where now $b\in C^\infty(M)$. For $\varsigma_B=-1$, the Killing spinor equation is $\nabla_X\epsilon = b X\cdot\epsilon$, so we are working with a generalisation of \emph{geometric} Killing spinors where the Killing number is allowed to be a function, known as the \emph{Killing function}, rather than a constant. This generalisation was considered for complex spinors on Riemannian spin manifolds in \cite{Lichnerowicz1987}, where it was shown in particular that the Killing function must be either real and constant or purely imaginary. That a real Killing function must be constant was already known from \cite{Hijazi1986}, the imaginary case was later studied in detail in \cite{Rademacher1990}. Note that our choice to work with the ``wrong'' sign ($+$) in the Clifford relation \eqref{eq:clifford-rel} effectively exchanges the roles of real and imaginary Killing number,\footnote{We can recover the standard treatment by working with complex Clifford algebras and spinors and replacing $\Gamma_\mu\mapsto i\Gamma_\mu$ which has the same effect on the Killing spinor equation as replacing $b\mapsto ib$.} so that our function $b$ being real-valued means that we are actually working in the ``imaginary'' case which will be verified when we consider the geometries supporting Killing ($D$-parallel) spinors. For $\varsigma_B=+1$, we have another generalisation of geometric Killing spinors known as \emph{skew-Killing spinors} \cite{Habib2012} (again, specifically the ``imaginary'' case).

 It remains to check the third condition from the definition. We will first derive some results using the integrability condition $R^D\epsilon=0$ for the existence of $D$-parallel spinors. We start with the following formula for the curvature of $D$ \cite[equation~(3.2)]{Beckett2024_2}:
\[
	R^D(X,Y)\epsilon = R(X,Y)\cdot\epsilon + \comm{\beta(X)}{\beta(Y)}\epsilon - (\nabla_X\beta)(Y)\epsilon + (\nabla_Y\beta)(X)\epsilon
\]
for all $X,Y\in\fX(M)$, $\epsilon\in\fS$, where $R$ is the Riemann curvature considered as a 2-form with values in skew-symmetric endomorphisms of $TM$.

For $\varsigma_B=-1$, we compute the following commutator:
\[
	\comm{\beta(X)}{\beta(Y)}\epsilon = b^2 (X\cdot Y - Y\cdot X)\cdot \epsilon = 2b^2 (X\wedge Y)\cdot \epsilon,
	\qquad \varsigma_B = -1,
\]
for $X,Y\in\fX(M)$ and $\epsilon\in\fS$, and $(\nabla_X\beta)(Y)=(\nabla_X b) Y\cdot \epsilon$. Our integrability condition is thus
\[
	R(X,Y)\cdot \epsilon + 2b^2 (X\wedge Y)\cdot \epsilon - ((\nabla_X b) Y - (\nabla_Y b) X)\cdot \epsilon = 0,
	\qquad \varsigma_B = -1.
\]
In the $\varsigma_B=+1$ case, we similarly find that
\[
	\comm{\beta(X)}{\beta(Y)}\epsilon = 2b^2 ((*X)\wedge (*Y))\cdot \epsilon = \varsigma 2b^2 (X\wedge Y)\cdot \epsilon,
		\qquad \varsigma_B = +1.
\]
The second equality is justified as follows, using the definition of the Hodge star and denoting the volume form by $\vol$
\[
	(*X)\wedge (*Y) = g(*X,Y)\vol = g(Y,*X)\vol = Y\wedge (**X) = -(**X)\wedge Y = \varsigma X\wedge Y
\]
since we can show that $*^2=-\varsigma \Id$ when acting on 1-vectors. The integrability equation is then
\[
	R(X,Y)\cdot \epsilon + \varsigma 2b^2 (X\wedge Y)\cdot \epsilon - ((\nabla_X b) *Y - (\nabla_Y b) *X)\cdot \epsilon = 0,
	\qquad \varsigma_B = +1.
\]

We use the integrability conditions above to show the following.

\begin{Lemma}\label{lemma:2d-dirac-current-preserve-beta}
	If $\epsilon$ is a spinor field such that $\nabla_X\epsilon=\beta(X)\epsilon$ then $\nabla_{\kappa_\epsilon} b = 0$.
\end{Lemma}

\begin{proof}
	First, note that $R(X,Y)\cdot\epsilon$ and $(X\wedge Y)\cdot\epsilon$ are proportional to $\vol\cdot\epsilon$, so in the $\varsigma_B=+1$ case, pairing the integrability condition with $\epsilon$ and using the fact that $\pair{\epsilon}{\vol\cdot\epsilon} = 0$ gives us
	\[
		(\nabla_X b)g(*Y,\kappa_\epsilon) - (\nabla_Y b)g(*X,\kappa_\epsilon) = 0
	\]
	for arbitrary $X,Y\in\fX(M)$. In the $\varsigma_B=-1$ case, we get the same equation by pairing with~${\vol\cdot\epsilon}$ and using $\pair{\epsilon}{\epsilon}=0$ along with $\vol\cdot\vol = -\varsigma\1$ and $\vol\cdot X = -*X$. Then, using the definition of the Hodge star operator and $\nabla_X b = \imath_{db}X$, the equation above is equivalent to~${\imath_{db}(X\wedge Y)\wedge\kappa_\epsilon = 0}$,
	which using a Leibniz rule is equivalent to $\imath_{\kappa_\epsilon}db(X\wedge Y) = 0$. Finally, since $X$, $Y$ are arbitrary, this gives us $\nabla_{\kappa_\epsilon}b=\imath_{\kappa_\epsilon}db=0$.
\end{proof}

\begin{Proposition}\label{prop:2d-real-KSA-exist}
	Let $(M,g)$ be a $2$-dimensional $($strongly$)$ spin manifold with signature $(0,2)$ or $(1,1)$, let $B$ be an admissible bilinear on the irreducible pinor module $S=\ssS$ $($of which there are two, distinguished by $\varsigma_B=\pm 1$$)$ and $\kappa$ the corresponding Dirac current. Let $\Sbundle$ denote the spinor bundle associated to $S$. Then if we define $\beta\in\Omega^1(M;\End\Sbundle)$ by equation \eqref{eq:2d-real-beta}, $D=\nabla-\beta$ is an admissible connection.
\end{Proposition}

\begin{proof}
	As already noted, conditions (1) and (2) of Definition~\ref{def:killing-spinor} are uniquely satisfied by equation \eqref{eq:2d-real-beta}. It remains to show condition~(3), namely that $\eL_{\kappa_\epsilon}\beta=0$ for all $\epsilon\in\fS_D$. For $\varsigma_B=-1$, we have
	\begin{align*}
		(\eL_{\kappa_\epsilon}\beta)(X)\zeta
			&= \eL_{\kappa_\epsilon}(\beta(X)\zeta) - \beta(\eL_{\kappa_\epsilon}X)\zeta - \beta(X)(\eL_{\kappa_\epsilon}\zeta)	\\
			&= \eL_{\kappa_\epsilon}(bX\cdot\zeta) - b(\eL_{\kappa_\epsilon}X)\cdot\zeta - bX\cdot (\eL_{\kappa_\epsilon}\zeta)	= (\eL_{\kappa_\epsilon}b)X\cdot\zeta	
	\end{align*}
	for all $X\in\fX(M)$ and $\epsilon,\zeta\in\fS$, where we have used the Leibniz rule twice. Similarly, for $\varsigma_B=+1$ we have $(\eL_{\kappa_\epsilon}\beta)(X)\zeta= (\eL_{\kappa_\epsilon}b)(*X)\cdot\zeta$. But then if $\epsilon\in\fS_D$, by Lemma~\ref{lemma:2d-dirac-current-preserve-beta}, we have $\eL_{\kappa_\epsilon}b=\nabla_{\kappa_\epsilon}b=0$, whence $\eL_{\kappa_\epsilon}\beta=0$ as required.
\end{proof}

Thus we have determined the precise form of the admissible connections with respect to the two different choices of Dirac currents. Since both maps are symmetric, these connections give rise to Killing \emph{super}algebras.

\subsection{Further insights from integrability}

We can obtain yet further results on the geometries which support Killing superalgebras using the integrability equations. If we have some non-zero $\epsilon\in\fS_D$, then since $R^D(X,Y)\epsilon=0$ the determinant of $R^D(X,Y)$ (as a spinor endomorphism) must vanish everywhere. Contracting the curvature with a Levi-Civita symbol for convenience, we have the following in a local frame:
\begin{equation}\label{eq:2d-D-curvature-comps}
	\varepsilon^{\mu\nu}R^D_{\mu\nu} =
		\begin{cases}
			\qty(\frac{1}{4}\varepsilon^{\mu\nu}\varepsilon^{\sigma\tau}R_{\mu\nu\sigma\tau} + 2b^2\varepsilon^{\mu\nu}\varepsilon_{\mu\nu})\Gamma_* - 2\varepsilon^{\mu\nu}\nabla_\mu b\Gamma_\nu
			\\
			\qquad=\varsigma\qty(\frac{1}{2}R + 4b^2)\Gamma_* - 2\varepsilon^{\mu\nu}\nabla_\mu b\Gamma_\nu
				& \text{for } \varsigma_B=-1,
			\\	
			\qty(\frac{1}{4}\varepsilon^{\mu\nu}\varepsilon^{\sigma\tau}R_{\mu\nu\sigma\tau} + \varsigma 2b^2\varepsilon^{\mu\nu}\varepsilon_{\mu\nu})\Gamma_* + \varsigma2\nabla^\mu b\Gamma_\mu
			\\
			\qquad= \qty(\varsigma\frac{1}{2}R + 4b^2)\Gamma_* + \varsigma 2\nabla^\mu b\Gamma_\mu
				& \text{for } \varsigma_B=+1,
		\end{cases}
\end{equation}
where in the second line of each case we have used the well-known fact that the Riemann tensor in 2 dimensions can be expressed in terms of the scalar curvature $R$
\[
	R_{\mu\nu\sigma\tau}=\tfrac{1}{2}R\qty(\eta_{\mu\sigma}\eta_{\nu\tau}-\eta_{\mu\sigma}\eta_{\nu\tau}) =\varsigma\tfrac{1}{2}R\varepsilon_{\mu\nu}\varepsilon_{\sigma\tau},
\]
and a combinatorial identity for $\varepsilon_{\mu\nu}$. Recalling our explicit descriptions of the Clifford algebras~$\Cl(0,2)$ and $\Cl(1,1)$ in terms of the Pauli matrices, we can compute the determinant in the local frame using the formula
$\det\qty(a\1 + b\sigma_1 + c\sigma_2 + d\sigma_3) = a^2 - b^2 - c^2 - d^2
$
for $a,b,c,d\in\CC$, giving us (we multiply by a sign for a slight simplification)
\begin{equation}\label{eq:2d-real-curv-det}
	\varsigma\det(\varepsilon^{\mu\nu}R^D_{\mu\nu}) =
		\begin{cases}
			\qty(\frac{1}{2}R + 4b^2)^2 - 4\norm{db}^2
				& \text{for } \varsigma_B=-1,
			\\	
			\qty(\frac{1}{2}R + \varsigma 4b^2)^2 - \varsigma 4\norm{db}^2
				& \text{for } \varsigma_B=+1.
		\end{cases}
\end{equation}

We would now like to set the expression above to zero and examine the resulting equations in the scalar curvature $R$ and the Killing number $b$. Let us first note that in Riemannian signature, the norm on differential forms is positive-definite, thus $\norm{db}^2\geq 0$ with equality if and only if $b$ is constant. A similar result holds in the Lorentzian case, but the statement is non-trivial since the norm is not positive-definite in this case.

\begin{Lemma}
	In Lorentzian signature, if $\fS_D\neq 0$, then $\varsigma_B\norm{db}^2\leq 0$, and $db$ is null at a point $p\in M$ if and only if $p$ is a critical point for $b$ or $\fS_D$ is spanned by a Killing spinor which is chiral at $p$.
\end{Lemma}

\begin{proof}
	We will prove the $\varsigma_B=-1$ case; the $\varsigma_B=+1$ case is entirely analogous. By Lemma~\ref{lemma:2d-dirac-current-preserve-beta}, $(db)^\sharp$ is orthogonal to $\kappa_\epsilon$ for all $\epsilon\in\fS_D$. By Corollary~\ref{coro:2d-dirac-causality}, for $\varsigma_B=-1$, $\kappa_\epsilon$ is everywhere either timelike or null, and it is null if and only if $\epsilon$ is chiral. Thus where $\kappa_\epsilon$ is timelike, $db$ must be either spacelike or zero, and where $\kappa_\epsilon$ is null, $db$ must be collinear with $\kappa_\epsilon$. This proves that~$\norm{db}^2\geq 0$. At a point $p$ where the inequality is saturated, we must have either $(db)_p=0$, whence $p$ is a critical point, or $(db)_p\neq 0$ is null and $\epsilon$ is chiral at $p$. Thus all Killing spinors must be chiral at $p$. But if there are two independent Killing spinors, there are linear combinations of such spinors which are not chiral at $p$, a contradiction.
\end{proof}

In particular, when $\fS_D\neq 0$, we may rearrange the equations obtained by setting the determinant in \eqref{eq:2d-real-curv-det} to zero to find
\begin{equation}\label{eq:2d-scalar-curv-formula}
	R =
		\begin{cases}
			\pm 4\abs{db} - 8b^2	& \text{for } \varsigma_B=-1,	\\
			\pm 4\abs{db} - \varsigma 8b^2		& \text{for } \varsigma_B=+1,
		\end{cases}
\end{equation}
where \smash{$\abs{db}=\sqrt{|\norm{db}^2|}$}.

Constraints of the type we have just derived are well-known in the literature on geometric Killing spinors and their generalisations, but we have arrived at this expression in a slightly non-standard way. The usual method is to derive the integrability condition from the \emph{Lichnerowicz formula} \cite{Hijazi2001,Lichnerowicz1963}
$\mp\slashed\nabla^2\epsilon + \Delta\epsilon = \frac{1}{4}R\epsilon$,
where $\slashed\nabla$ is the Dirac operator (locally $\Gamma^\mu\nabla_\mu$), $\Delta$ the Laplace operator (locally $g^{\mu\nu}\nabla_\mu\nabla_\nu$) $R$ the scalar curvature, and the sign on the first term is the opposite to that in the Clifford relation~\eqref{eq:clifford-rel}, and this identity holds for all $\epsilon\in\fS$. Indeed, in two dimensions it is completely equivalent to $R^D\epsilon=0$.

Finally, we recall that for $\varsigma_B=-1$, Killing spinors are nothing but (generalised) geometric Killing spinors with imaginary Killing function; indeed, for the Riemannian ($\varsigma=+1$) case, a~similar equation to \eqref{eq:2d-scalar-curv-formula} is found in \cite[equation~(5)]{Rademacher1990}, though there it is derived by different means. That work also classifies the possible Riemannian geometries supporting such spinors. In the more well-known case where $b$ is constant, we find that $R=-8b^2$, whence the geometry is (at least locally) hyperbolic space $H^2$.\footnote{If we had worked with a negative-definite metric and the ``correct'' sign in the Clifford algebra, we would have found $R=2\lambda^2$. However, in negative-definite signature this is the correct sign for the curvature of $H^2$ when~${b\in \RR}$, consistent with the interpretation which arises using our original sign convention.}

\subsection{Maximally supersymmetric case}

We call $(M,g,D)$ \emph{maximally supersymmetric} if $\dim\fS_D=\dim S=2$. In this case, the values of $D$-parallel spinors span the fibre of $\Sbundle$ at every point of $M$, and since $R^D$ annihilates these values, we must have $R^D=0$ identically, and we can use this condition to identify $(M,g)$ up to local isometry; let us therefore assume that $M$ is simply connected. From our local expression~\eqref{eq:2d-D-curvature-comps}, we see that $b$ must be constant and we have $R=-8b^2 $ (for $\varsigma_B=-1$) or~${R=-\varsigma 8b^2}$ (for~${\varsigma_B=+1}$). Thus the maximally supersymmetric $\varsigma_B=-1$ case is precisely the classic geometric Killing spinor regime with imaginary Killing number as discussed above. For real $b$, the scalar curvature is constant and negative, so the geometry must be hyperbolic in Riemannian signature and anti-de Sitter in Lorentzian signature. For $\varsigma_B=+1$, the maximally supersymmetric Riemannian geometry is hyperbolic again, while the maximally supersymmetric Lorentzian geometry is de Sitter. This agrees with the results of Section~\ref{sec:2d-max-susy-subdef}, where the same geometries were identified as homogeneous spaces for the even parts of filtered deformations of $\fs$. Indeed, those deformations are precisely the Killing superalgebras of these maximally supersymmetric backgrounds.\looseness=1

\section{Summary of results}

The results of the calculations presented in Sections \ref{sec:2d-spencer} and \ref{sec:2d-ksa} are summarised in Table~\ref{table:2d-spencer}.

\begin{table}[ht]
\centering
	\caption{Summary of results, including the Spencer $(2,2)$-cohomology group, a characterisation of the Killing spinors for the admissible connection and the (non-trivial) maximally supersymmetric geometries.}
	\label{table:2d-spencer}
\vspace{1mm}

	\setlength{\extrarowheight}{1.2ex}
	$\begin{array}{c|c|c|c|c|c}
		\text{Signature}	& S		& B		& \cH^{2,2}	& \text{Killing spinors}	& \text{Max. SUSY geom.}	\\
		\hline
		\multirow{2}{*}{(0,2)}	
		& \ssS_1=\ssS	& \varsigma_B = +		& \RR	& \nabla_X\epsilon=b(*X)\cdot \epsilon	& H^2	\\
		& \ssS_1=\ssS 	& \varsigma_B = -		& \RR	& \nabla_X\epsilon=bX\cdot \epsilon		& H^2	\\
		\hline
		\multirow{3}{*}{(1,1)}	
		& \ssS=\ssS_+\oplus\ssS_-	& \varsigma_B = +	& \RR	& \nabla_X\epsilon=b(*X)\cdot \epsilon	& dS_2	\\
		& \ssS=\ssS_+\oplus\ssS_-	& \varsigma_B = -	& \RR	& \nabla_X\epsilon=bX\cdot \epsilon		& AdS_2	\\
		& \ssS_\pm					& -				& 0		& \nabla_X\epsilon=0					& -	
	\end{array}$
\end{table}

We note that we have treated only signatures $(0,2)$ and $(1,1)$ but considered all independent admissible Dirac currents on the pinor module $\ssS$, of which there are two in each signature distinguished by the symmetry $\varsigma_B$ of the associated admissible bilinear (constructed in Section~\ref{sec:2d-dirac-currents}). We also considered only the minimal ($S=\ssS$) case in signature $(0,2)$ and minimal chiral ($S=\ssS_+$) and non-chiral ($S=\ssS=\ssS_+\oplus\ssS_-$) cases in signature $(1,1)$.

The chiral case was not explicitly discussed in Section~\ref{sec:2d-max-susy-subdef} or Section~\ref{sec:2d-ksa} but is included in the table; it essentially reduces to the non-chiral case (for either choice of $B$) with $b=0$ and a space of Killing spinors (or odd subspace of the superalgebra) which is the one-dimensional span of a~chiral spinor. The bilinear $B$ is trivial when restricted to this subspace, but the Dirac current~$\kappa$ is not. The Killing spinors are parallel spinors (with their currents being parallel null vectors), and the notion of ``maximal supersymmetry'' is somewhat vacuous, with the possible Killing superalgebras in this case being nothing but graded subalgebras of the 2-dimensional minimal chiral Poincar\'e superalgebra.

The $(2,0)$ case, omitted here so as not to complicate the presentation with the need to discuss quaternionic (symplectic Majorana) spinors -- as well as $N$-extended supersymmetry -- will be treated in future work.

\subsection*{Acknowledgements}

The author would like to thank Jos\'e Figueroa-O'Farrill, under whose supervision this work was done, for his guidance and patience. Thanks also to Andrea Santi, James Lucietti, C.S.~Shahbazi and the sorely missed Paul de Medeiros for many enlightening conversations and helpful comments. Finally, thanks to the anonymous referees, whose comments and suggestions greatly improved the quality and clarity of this work.

This work was carried out with scholarship funding from the Science and Technologies Facilities Council (STFC) and the School of Mathematics at the University of Edinburgh.

This work previously appeared as part of the author's Ph.D.~Thesis \cite{Beckett2024} but has not been published elsewhere. Some changes in notation and terminology have been made for this version to aid comparison with other work, and there are also some minor changes in exposition and corrections.

\pdfbookmark[1]{References}{ref}
\LastPageEnding

\end{document}